\documentclass[12pt]{amsart}
\usepackage{amssymb,amsmath}
\usepackage{graphicx,color}

\usepackage[all]{xy}
\usepackage{amscd}
\usepackage{hyperref}

\usepackage[active]{srcltx}

\nonstopmode

\def\cocoa{{\hbox{\rm C\kern-.13em o\kern-.07em C\kern-.13em o\kern-.15em A}}}

\DeclareMathOperator{\Hom}{Hom}
\DeclareMathOperator{\ann}{ann}
\DeclareMathOperator{\Soc}{Soc}
 
\newtheorem{theorem}{Theorem}[section]
\newtheorem{lemma}[theorem]{Lemma}
\newtheorem{proposition}[theorem]{Proposition}
\newtheorem{corollary}[theorem]{Corollary}
\theoremstyle{definition}
\newtheorem{definition}[theorem]{Definition} 
\newtheorem{remark}[theorem]{Remark}
\newtheorem{example}[theorem]{Example}

\definecolor{MyDarkGreen}{cmyk}{0.7,0,1,0}
\definecolor{orange}{cmyk}{0,0.5,0.7,0}

\numberwithin{equation}{section}
\title[On the Weak Lefschetz Property for Artinian Gorenstein Algebras]{On the Weak Lefschetz Property for Artinian Gorenstein Algebras of Codimension Three}

\author[M.\  Boij]{Mats Boij${}^{1}$}
\address{Department of Mathematics, KTH - Royal Institute of Technology, S-100 44 Stockholm, Sweden}
\curraddr{}
\email{boij@kth.se}

\author[J.\ Migliore]{Juan Migliore${}^{2}$}
\address{
  Department of Mathematics, University of Notre Dame, Notre Dame, IN
  46556, USA}
\email{Juan.C.Migliore.1@nd.edu}

\author[R.\ M.\ Mir\'o-Roig]{Rosa M.\ Mir\'o-Roig${}^{3}$}
\address{Facultat de Matem\`atiques, Department d'\`Algebra i
 Geometria, Gran Via des les Corts Catalanes 585, 08007 Barcelona, Spain}
\email{miro@ub.edu}

\author[U.\ Nagel]{Uwe Nagel${}^{4}$}
\address{Department of Mathematics,
University of Kentucky, 715 Patterson Office Tower,
Lexington, KY 40506-0027, USA}
\email{uwe.nagel@uky.edu}

\author[F.\ Zanello]{Fabrizio Zanello}
\address{Department of Mathematical Sciences, Michigan Technological
  University, Houghton, MI 49931, USA {\small and} Department of
  Mathematics, MIT, Cambridge, MA 02139-4307, USA}
\email{zanello@math.mit.edu}

\thanks{\noindent Printed \today \\
${}^{1}$ Part of this work was done while this author was a Research
Member at the MSRI.
\\
${}^{2}$ Part of the work for this paper was done while this
author was sponsored by the National Security Agency under Grant
Numbers H98230-09-1-0031 and H98230-12-1-0204, and by the Simons Foundation under grant \#208579.\\
${}^{3}$  Part of the work for this paper was done while this
author was sponsored by MTM2010-15256.\\
${}^{4}$ Part of the work for this paper was done while this author was
sponsored by the National Security Agency under Grant
Numbers H98230-09-1-0032 and H98230-12-1-0247, and by the Simons Foundation under grant \#208869.
}

\begin{document}


\begin{abstract}
We study the problem of whether an arbitrary codimension three graded artinian Gorenstein algebra has the Weak Lefschetz Property.  We reduce this problem to checking whether it holds for all compressed Gorenstein algebras of odd socle degree.  In the first open case, namely Hilbert function $(1,3,6,6,3,1)$, we give a complete answer in every characteristic by translating the problem to one of studying geometric aspects of certain morphisms from $\mathbb P^2$ to $\mathbb P^3$, and Hesse configurations in $\mathbb P^2$.

\end{abstract}

\maketitle

\section{Introduction}
 The Weak Lefschetz Property (WLP) for an artinian graded algebra $A$
 over a field $k$ simply says that there exists a linear form $L$ that
 induces, for each $i$, a multiplication $\times L : [A]_i
 \longrightarrow [A]_{i+1}$ that has maximal rank, i.e. that is either
 injective or surjective.  At first glance this might seem to be a
 simple problem of linear algebra, but instead it has proven to be
 extremely elusive, even in the case of very natural families of
 algebras.  Many authors have studied these problems from many
 different points of view, applying tools from representation theory,
 topology, vector bundle theory, plane partitions, splines,
 differential geometry, among others (cf. \cite{BK,GIV,HSS,HMMNWW,KRV,KV,MMO,St3}). The role of the characteristic of $k$ in this problem has also been an important, if only superficially understood, aspect of these studies.  

One of the most interesting open problems in this field is whether all
codimension $3$ graded artinian Gorenstein algebras have the WLP in
characteristic zero.  In the special case of codimension $3$ complete intersections, a positive answer was obtained in characteristic zero in \cite{HMNW} using the Grauert-M\"ulich theorem.  For positive characteristic, on the other hand, only the case of monomial complete intersections has been studied (cf. \cite{BK2,CGJL,Cook,CN1,CN2,LZ}), applying many different approaches from combinatorics.  

For the case of codimension $3$ Gorenstein algebras that are not necessarily complete intersections, it is known that for each possible Hilbert function an example exists having the WLP (in \cite{harima} this is proved in a much more general setting).  Some partial results are given in \cite{MZ} to show that for certain Hilbert functions, all such Gorenstein algebras do have the WLP.  But the general case remains completely open.

The present paper has two goals.  First, we give a  reduction of this problem to a very specific type of Gorenstein algebra.  That is, we show that in order to prove that the WLP holds for {\em all} artinian Gorenstein algebras of codimension $3$, it is enough to prove that it holds for all compressed artinian Gorenstein algebras of odd socle degree; that is, for  artinian Gorenstein algebras  having maximal Hilbert function
\[\left(
\textstyle 1,3,6,\dots,\binom{t-1}2,\binom{t+2}{2}, \binom{t+2}{2},\binom{t-1}2, \dots, 6, 3, 1
\right).
\]
The first case, $(1,3,3,1)$, follows from a result in \cite{MZ}.

The second goal of this paper is to show that a solution of the first
open case, namely Hilbert function $(1,3,6,6,3,1)$, already involves
subtle geometric properties of certain morphisms between projective
spaces and certain classical configurations of points known as Hesse
configurations.  We use these ideas to show that the WLP always holds
for $(1,3,6,6,3,1)$ in characteristic zero, and that in positive
characteristic the WLP continues to hold except when the
characteristic is $3$ and, after change of variables, the ideal is of
the form $(x^2y, x^2z, y^3, z^3, x^4+y^2z^2)$.  As a corollary we get
that in characteristic zero all artinian Gorenstein algebras of socle
degree at most $5$ also have the Strong Lefschetz Property.

Note that any solution to this problem will necessarily involve
methods that are special to a polynomial ring in three variables (or
from a geometric point of view, a consideration of the projective
plane), since it is well-known that this problem has a negative answer
in four or more variables.  Attempts to solve this problem in the past
have involved many approaches, including the use of theorems of
Macaulay and of Green, the use of syzygy bundles, etc.  The idea of
using a basis for the vector space of forms of initial degree in the
ideal to define a morphism from $\mathbb P^2$ to a suitable projective
space (when this basis defines a base point-free linear system) has
also been used in the past.  In this paper we give a connection
between this latter approach and the existence of very special
configurations (Hesse configurations) of points in the plane in the
first open case, namely $(1,3,6,6,3,1)$.  Our tools are the geometry
of the morphism and its image, as well as liaison theory.  Moreover, as an important ingredient we use the Buchsbaum-Eisenbud
Structure Theorem \cite{BE} which is special to the three variable case. 
It is our belief that our method has much more to give. 
It has recently been established that the failure of WLP is related to
morphisms with unexpected properties \cite{MMO}, which has been used,
e.g., in \cite{CN2} and \cite{GIV}. 
Indeed, although we conjecture that ultimately a positive answer will
arise in characteristic zero, we think that our approach could also
indicate where to look for a counter-example.  Thus we present it not
as a finished package, but as a new tool to attack the general case.


\section{Injectivity and reduction to the compressed case}

Let $k$ be an algebraically closed field. Given a graded artinian algebra $A=S/I$ where $S=k[x_1,x_2,\dots,x_n]$ and $I$ is a homogeneous ideal of $S$,
we denote by $H_A:\mathbb{Z} \longrightarrow \mathbb{Z}$ with $H_A(j)=\dim _k[A]_j$
its Hilbert function. Since $A$ is artinian, its Hilbert function is
captured in its {\em $h$-vector} $h=(h_0,h_1,\dots ,h_e)$ where $h_i=H_A(i)>0$ and $e$ is the last index with this property. The integer $e$ is called the {\em socle degree of} $A$.
We will use the notion \emph{codimension} for $h_1$, even though in some
cases, \emph{embedding dimension} would be a more correct notion. 

Recall that the graded $k$-algebra $A$ is Gorenstein if its socle $\Soc(A)=\{a\in A \mid a\cdot (x_1,x_2,\dots,x_n)=0 \}$ is $1$-dimensional, i.e. $\Soc(A)\cong k(-e)$.  Its $h$-vector is symmetric, i.e. $h_i=h_{e-i}$ for all $i$.

\begin{definition}
Let $A=S/I$ be a graded artinian $k$-algebra. We say that $A$ has the {\em Weak Lefschetz Property} (WLP)
if there is a linear form $L \in [A]_1$ such that, for all
integers $i\ge0$, the multiplication homomorphism
\[
\times L: [A]_{i} \to [A]_{i+1}
\]
has maximal rank, i.e.\ it is injective or surjective.  We say that  $A$ has the \emph{Strong Lefschetz Property} (SLP)
if there is a linear form $L \in [A]_1$ such that for any two indices $i\le j$, the multiplication homomorphism $\times L^{j-i}: [A]_{i} \to [A]_j$ has maximal rank.
\end{definition}

In this section we will exploit an idea from Hausel's proof of the
following theorem. Recall that a \emph{pure $O$-sequence} is the
Hilbert function of a \emph{monomial} artinian level algebra, or
equivalently, the rank function of a finite monomial order ideal  (see
e.g. Stanley's seminal paper \cite{sta1}, where both pure
$O$-sequences and level algebras were  introduced). 

\begin{theorem} \emph{(Hausel \cite{hausel})} Let $h=(h_0,h_1,\dots ,h_e)$ be a pure $O$-sequence. Then the first half of $h$ is differentiable, and $h_i\le h_j$ for any $i\le j\le e-i$. 
\end{theorem}

The idea is that by Macaulay duality, the ideal $I$ of an artinian level algebra of type
$t$ is
the intersection of Gorenstein ideals $I_1,I_2,\dots,I_t$, and that we have an injective
homomorphism 
$$
p\colon S/I \longrightarrow \prod_{i=1}^t S/I_i.
$$
If the multiplication by a form $f$ on each factor $S/I_i$ is
injective, then so is the multiplication on $S/I$. In Hausel's case,
the Gorenstein algebras are monomial complete intersections and they
all satisfy the SLP, so powers of linear forms give injective
multiplication in the appropriate degrees on the monomial level
algebra.

Since we will use this idea extensively, we formalize it into the
following remark. 
\begin{remark} \label{injle} Let $V$ and $W$ be vector spaces over the same field $k$, and let $V_1,\dots,V_t$ and $W_1,\dots,W_t$ be subspaces of $V$ and $W$, respectively. Set $\widetilde{V}=\bigcap_{i=1}^t V_i$ and $\widetilde{W}=\bigcap_{i=1}^t W_i$. Suppose there exists a vector space homomorphism $\phi: V \longrightarrow W$ such that all of the induced homomorphisms $\phi_i: V/V_i \longrightarrow W/W_i$ are well-defined and injective. Then the induced homomorphism $\widetilde{\phi}: V/\widetilde{V} \longrightarrow W/\widetilde{W}$ is also well-defined and injective.
\end{remark}

In this section we will use Remark~\ref{injle} to prove that in order for all
Gorenstein algebras of codimension $3$ to enjoy the WLP, it suffices to
prove the property for those that are \emph{compressed} and of
\emph{odd socle degree}.  

As we will see below, Remark~\ref{injle} also  connects, in a new way, the Weak and the Strong Lefschetz Properties with some of the  combinatorial  features of a Hilbert function, such as unimodality, differentiability, and flawlessness. 

For instance, we will show that, if the WLP holds for all codimension $3$ Gorenstein algebras, then all  {level} Hilbert functions of the same codimension are \emph{differentiable} throughout their first half (i.e., their first difference is an $O$-sequence). If, moreover, all Gorenstein algebras enjoy the SLP (in fact, a bit less), then all level Hilbert functions $h=(1,h_1,\dots,h_e)$ satisfy the inequalities $h_i\le h_j$, for all indices $i\le j\le e-i$. In particular, they are \emph{flawless}; i.e., $h_i\le h_{e-i}$ for all $i\le e/2$.

Next, we  discuss some of the consequences of Remark~\ref{injle} that concern level and Gorenstein algebras. The first of these applications greatly simplifies the study of the WLP in the codimension $3$ Gorenstein case, which is the main topic of this paper. 

Recall that an artinian level algebra $A=S/I$ of socle degree $e$ and type $t$ is said to be \emph{compressed} if its Hilbert function is given by:
\[
h_A (i) = \min \{ \dim_k S_i, t \cdot \dim_k S_{e-i} \}.
\] 

In other words, the Hilbert function of $A$ is as large as possible in each degree, for it grows maximally both from the right and from the left. While it is not hard to see that the displayed formula for $h_A$ is an upper-bound, it is nontrivial to show that the bound is actually achieved. In fact, more is true: there exists an irreducible parameter space for which that bound is achieved on a (nonempty) Zariski-open subset.

Compressed algebras were first defined and studied (in a slightly more general context than level algebras) by A. Iarrobino \cite{Iarrobino-compr}, using \emph{Macaulay's inverse systems} (we refer to \cite{Ge,IK} for details on the theory of inverse systems). See also \cite{FL}, where R. Fr\"oberg and D. Laksov reproved some of Iarrobino's results by means of a direct approach, and the fifth author's \cite{Za1,Za2}, for a  generalization of the concept to arbitrary artinian algebras.

We shall also  need the following lemma. In its proof, $\langle
H\rangle$ will  denote the {inverse system module} generated by a
homogeneous form $H\in E=\bigoplus_{i=0}^\infty \Hom_k(S_i,k)$, and $\ann(\langle H\rangle)$ the annihilator of $\langle H\rangle$ in $S=k[x_1,\dots,x_r]$. 

\begin{lemma}\label{lemma:odd} Fix $e\ge 3$ odd. If some  Gorenstein algebra of socle degree $e$ fails  the WLP, then the same is true for some \emph{compressed} Gorenstein algebra of socle degree $e$.
\end{lemma}

\begin{proof}
Let $F\in E$ be an element of odd degree $e=2d+1$. We will show that, if
$A=S/\ann(\langle F\rangle )$ does not satisfy the WLP, then there
exists an element $G$ of degree $e$ such that $B=S/\ann(\langle G\rangle)$
is compressed and does not satisfy the WLP. Suppose the Hilbert
function of $A$ is $(1,h_1,\dots,h_{2d+1})$. In order to verify
whether or not the WLP holds for $A$, it suffices to check the
injectivity of the homomorphism $\times L \colon A_d\longrightarrow A_{d+1}$, for a general linear form $L$. 

By \cite{Iarrobino-compr}, there exist elements $H_1,H_2,\dots,H_s\in
E$ of degree $e$ such that if $G=F+\sum_{i=1}^s H_i$ and $I=\ann(\langle G
\rangle)$, then $B=S/I$ has Hilbert function
$\min\{\dim_kS_i,h_i+s\}$, for $i=0,1,\dots,d$. If $s=\dim_kS_d-h_d$,
we obtain that $B$ is a compressed Gorenstein algebra. On the other
hand, the rank of the homomorphism $\times L\colon B_d\longrightarrow B_{d+1}$
is at most $s$ higher than the rank of $\times L\colon A_d\longrightarrow A_{d+1}$. Thus, if $A$ fails to have the WLP, so does $B$.
\end{proof}

\begin{corollary}\label{reductiontocompressedodd}
Let $r$ and $e$ be positive integers,  with $e$ even. If all artinian \emph{compressed} Gorenstein algebras of codimension $\le r$ and socle degree $e-1$  have the WLP, then all artinian Gorenstein  algebras of codimension $\le r$ and socle degree $e$ have the WLP.

In particular, since all codimension $2$ algebras have the WLP (see  \cite{HMNW,MZ0}), if {all} codimension $3$ artinian \emph{compressed} Gorenstein  algebras \emph{of odd socle degree} have the WLP, then \emph{all} codimension $3$ artinian Gorenstein  algebras  have the WLP.
\end{corollary}

\begin{proof}
By Lemma~\ref{lemma:odd}, it is enough to show that if $A=\bigoplus_{i=0}^eA_i$ is a Gorenstein algebra of codimension $r$ and even socle degree $e$ failing the WLP, then the WLP also fails for some  Gorenstein  algebra of  socle degree $e-1$ and codimension $\le r$. 

The  truncation $B=\bigoplus_{i=0}^{e-1}B_i=S/I$ of $A$ is clearly a
level algebra also failing the WLP. Since $B$ has type
$t=\dim_kB_{e-1}$, we have that $I$ can be expressed as the intersection of $t$ ideals  of $R$, say $I^{(1)},\dots,I^{(t)}$, such that all algebras $B^{(i)}=S/I^{(i)}$ are Gorenstein of socle degree $e-1$ and codimension at most $r$. 

Since the WLP for Gorenstein algebras only needs to be checked in the middle degree, an application of Remark~\ref{injle} now shows that at least one of the $B^{(i)}$ must fail the WLP. This  completes the proof of the corollary.
\end{proof}

The following corollary shows two important consequences of Remark~\ref{injle}, which connect in a new fashion the Lefschetz Properties and the combinatorics of level Hilbert functions. Proving the next two   results unconditionally, for instance for level Hilbert functions of codimension $3$ and characteristic zero, would be extremely interesting.

\begin{corollary}
\begin{itemize}
\item[(1)] Assume all  artinian Gorenstein algebras of socle degree $e$ and codimension $\le r$ enjoy the WLP. Then the first half of any artinian level Hilbert function of socle degree $e$ and codimension $\le r$ is differentiable.
\item[(2)] Assume all artinian Gorenstein algebras of socle degree $e$ and codimension $\le r$ enjoy the SLP. Then any  artinian level Hilbert function  $h=(h_0,h_1,\dots ,h_e)$, with $h_1\le r$, also satisfies the inequalities $h_i\le h_j$, for all $i\le j\le e-i$. In particular, $h$ is flawless.
\end{itemize}
\end{corollary}

\begin{proof}
(1) Let $A=S/I$ be a level algebra having Hilbert function  $h$, where $I$ is the intersection of $t=h_e$   Gorenstein ideals $I^{(1)},\dots,I^{(t)}$. Since all Gorenstein algebras $B^{(a)}=S/I^{(a)}$ have the WLP, if $L\in R$ is a general form, then multiplication by $L$  is injective between any two consecutive degrees of the first half of any of the $B^{(a)}$. Thus, Remark~\ref{injle} easily gives us that the injectivity of  $\times L$ is inherited by $A$ throughout its first half; in particular, $h$ is differentiable in those degrees, as desired.

(2) Since  Gorenstein Hilbert functions are symmetric, if all algebras $B^{(a)}$ have the SLP, then for all indices $i\le j\le e-i$ and $a=1,\dots, t$, and any general linear form $L\in S$, we immediately have that multiplication by $L^{j-i}$ is injective from $B^{(a)}_i$ to $B^{(a)}_j$. The result now similarly follows from Remark~\ref{injle}.
\end{proof}

\begin{remark}
We only remark here that, in the previous corollary, the SLP can in fact be replaced by a weaker but also  natural condition, called the ``Maximal Rank Property'' (see \cite{MM} for details).
\end{remark}

Recall that a numerical sequence is  \emph{unimodal} if it does not strictly increase after a strict decrease.

\begin{example}
We know from  \cite{HMNW} that, when $k$ has characteristic zero,  every artinian complete intersection quotient of $S=k[x,y,z]$ has the WLP. Let $A=S/I$ be a type $t$ level algebra, such that $I$ can be written as the intersection of $t$ complete intersection artinian ideals. Then, from Remark~\ref{injle} and the proof of Corollary~\ref{reductiontocompressedodd}, it easily follows that the Hilbert function $h_A$ of $A$ is \emph{differentiable} throughout its first half. Notice that, moreover, if $h_A$ is {unimodal} with a peak in its middle degree, then by \cite{MMN2}, $A$ also enjoys the WLP.
\end{example}

As a final  illustration of the scope of Remark~\ref{injle}, we record the following result, whose proof is along the same lines as the above, hence will be omitted.

\begin{corollary}\label{uniWLP} Let $A$ be an
  artinian Gorenstein algebra of socle degree $e$ with
  \emph{nonunimodal} Hilbert function. Then the general  Gorenstein quotient of $A$ of socle degree $e-1$ fails the WLP.
\end{corollary}

\begin{remark}
\begin{itemize}
\item[(i)] For instance, we know from \cite{MNZ,Za_IP} (see also
  \cite{MNZ1} for a broad generalization) that there exist ``many''
  Gorenstein algebras having a nonunimodal Hilbert function of the
  form $H=(1,r,a,r,1)$, where $a<r$. In fact, as conjectured by
  Stanley \cite{St2}, in \cite{MNZ} three of the present authors
  proved that $\lim_{r\rightarrow \infty} \frac{f(r)}{ r^{2/3}}=
  6^{2/3}$, where $f(r)$ is the least possible value of $a$ such that
  $H$ is Gorenstein; in \cite{Za_IP}, the fifth author then showed
  that $H$ is Gorenstein for all values of
  $a=f(r),\dots,\binom{r+1}{2}$. 

Therefore, by Corollary~\ref{uniWLP}, any Gorenstein algebra with Hilbert function $H$ gives rise to a number of new Gorenstein algebras of socle degree $3$ (and codimension $r' < a$) that do not enjoy the WLP.

\item[(ii)] As the example in part (i) shows, Remark~\ref{injle} also provides a new connection between the WLP and unimodality.  Indeed, in Corollary \ref{uniWLP} we cannot hope to replace the failure of WLP by non-unimodality in general since the above-mentioned Gorenstein algebras of socle degree $3$ have Hilbert function $(1,r',r',1)$.
\end{itemize}
\end{remark}


\section{Hesse configurations and the WLP}

From now on, we let $S=k[x,y,z]$. In \cite{HMNW}, it was shown that, if  the characteristic of $k$ is zero, every complete intersection $A=S/(F_1, F_2, F_3$) has the WLP. (This is false in positive characteristic.) We would like to extend this result and  know whether the same holds for {\em any} codimension $3$ artinian Gorenstein algebra $A=S/I$. It is known that in characteristic zero the WLP holds for a nonempty  open subset of the codimension $3$ artinian Gorenstein graded algebras with fixed  Hilbert function (cf. \cite{harima}); however showing it for {\em all} codimension $3$ artinian Gorenstein graded algebras has proved very elusive so far. Corollary~\ref{reductiontocompressedodd} shows that it is enough to prove it in the case of compressed Gorenstein algebras of odd socle degree.   The goal of this section is to prove the first open nontrivial case; namely, that all codimension $3$ artinian Gorenstein algebras with $h$-vector $(1, 3, 6, 6, 3, 1)$ do have WLP in every characteristic, except only when the characteristic is $3$ and the ideal has a very precisely given form, up to change of variables.

\begin{definition}
A {\em Hesse configuration}  in $\mathbb P^2$ is a complete intersection, $Y$, of two cubics with the property that the line containing any length two subscheme of $Y$ actually meets $Y$ in a length $3$ subscheme.
\end{definition}

\begin{remark} \label{hesse picture}
Note that we do not assume above that a Hesse configuration is reduced. Hesse configurations arise as the set of flex points on a smooth cubic curve over the complex numbers, and also as the affine plane over the field of three elements.    
\end{remark}

\begin{figure}[hbt]
  \centering
  \includegraphics[width=2in,angle=90]{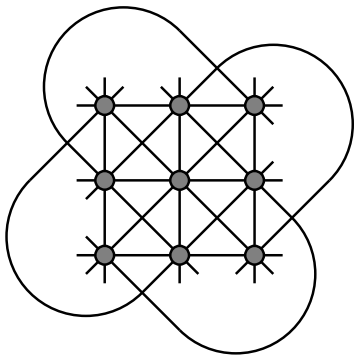}
  \caption{Hesse configuration}
  \label{fig:hesse}
\end{figure}

In our study of Gorenstein algebras $A$ of $h$-vector $(1,3,6,6,3,1)$, we will make use of a length $7$ subscheme $X$ of $\mathbb{P}^2$ defined by a submatrix of the skew-symmetric matrix associated to $A$. We begin with a careful study of these subschemes. 

\begin{proposition} \label{tfae}
  Let $X$ be a zero-dimensional subscheme of length $7$ in $\mathbb
  P^2$ not lying on a conic, i.e., having $h$-vector $(1,2,3,1)$, such
  that $I_X$ has three minimal generators (hence all of degree three).  Let $f$
  be a form of degree $3$ which is a non-zerodivisor on $S/I_X$ and
  let $J=I_X+(f)$. 
  Let
   $$
  M=\begin{bmatrix}
    \ell_1&q_1\\
    \ell_2&q_2\\
    \ell_3&q_3\\
  \end{bmatrix}
  $$
be the Hilbert-Burch matrix for $X$, where 
  $\ell_1,\ell_2,\ell_3$ are linear forms and  $q_1,q_2,q_3$ are quadratic forms. Then the following are equivalent:

\begin{itemize}
\item[(a)]  the forms $\ell_1, \ell_2, \ell_3$ are linearly independent; 

\item[(b)] $X$ does not have any subscheme of length $6$ on a conic;

\item[(c)] $X$ does not contain a subscheme that is a complete intersection of type $(2,3)$; 

\item[(d)] there is a form $g$ of degree $4$ such that
  $J+(g)$ is a Gorenstein ideal with Hilbert function
  $(1,3,6,6,3,1)$.
\end{itemize}
\end{proposition}

\begin{proof}
Clearly (b) implies (c).  A subscheme, $Z$, of length $6$ on a conic is a complete intersection unless the forms in $I_Z$ of degree $3$ in the ideal all have a common factor.  But $I_Z$ contains $I_X$, and by assumption the forms of degree $3$ in $I_X$ do not have a common factor, since they generate $I_X$.  Hence (c) implies (b).

 The linear
  forms span either a $2$-dimensional 
  or a $3$-dimensional space. If they span a $2$-dimensional space, we
  can assume that $\ell_3=0$ and we get a subscheme of length $6$ on the
  conic defined by $q_3$. If $Z \subseteq X$ is a subscheme of length $6$
  on a conic, we get that $I_Z/I_X$ has Hilbert function
  $(0,0,1,1,1,\dots)$, so we have a quadratic generator $q$ of $I_Z$ and
  two linear forms $\ell'$ and $\ell''$ such that $q\ell'$ and $q\ell''$
  are in $I_X$. Since these elements have degree three, they are among the
  minimal generators of $I_X$ and they have a common factor. This forces
  the three linear forms of the Hilbert-Burch matrix to be linearly
  dependent. Thus (a), (b) and (c) are equivalent.
  
  Now look at the minimal free resolution of $S/J$. Since $f$ is a non-zerodivisor,
  we get the resolution of $S/J$ from the resolution of $S/I_X$ and the
  Betti diagram will be
  $$
  \begin{matrix}
    &0&1&2&3\\
    \text{total:}&1&4&5&2\\
    \text{0:}&1&\text{.}&\text{.}&\text{.}\\
    \text{1:}&\text{.}&\text{.}&\text{.}&\text{.}\\
    \text{2:}&\text{.}&4&1&\text{.}\\
    \text{3:}&\text{.}&\text{.}&1&\text{.}\\
    \text{4:}&\text{.}&\text{.}&3&1\\
    \text{5:}&\text{.}&\text{.}&\text{.}&1\\
  \end{matrix}
  $$
  where the Hilbert-Burch matrix $M$ will occur as a submatrix of the last
  homomorphism.  From the last column, we see that we have one socle element in
  degree four. When dualizing, the dual module
  $(S/J)^\vee=\Hom_k(S/J,k)$, has two minimal generators, one in
  degree $-4$ and one in degree $-5$. 

  If the linear forms of $M$
  are linearly independent, the generator in degree $-5$ generates
  everything in degree $-3$ and the generator in degree $-4$ is
  multiplied into the module generated by the other element by the
  maximal ideal. 
  This means that when we choose $g$ to be the
  generator of the socle in degree $4$, we get a Gorenstein quotient $S/(J+(g))$
  with $h$-vector $(1,3,6,6,3,1)$. 
  
  If the linear forms of $M$ only span a $2$-dimensional space, this
  means that the dual generator corresponding to the socle element of
  degree $4$ is not multiplied into the module generated by the other
  generator by all linear forms. Hence, the dual  generator of degree
  $-5$ corresponding to the socle element of degree $5$ does not 
  generate all forms of degree $-3$ in the dual and we cannot get a
  Gorenstein quotient of $S/J$ with $h$-vector $(1,3,6,6,3,1)$.
\end{proof}

We want to show that any graded artinian Gorenstein algebra with
$h$-vector $(1,3,6,6,3,1)$ has the WLP.  First we note that, if the
WLP fails, then the rank of the corresponding multiplication homomorphism is off by exactly one.

\begin{lemma} \label{rank}
Let $A$ be a graded artinian  algebra with $h$-vector $(1,3,6,6,
\dots)$.  Let $L$ be a general linear form.  Then $\times L : [A]_i
\rightarrow [A]_{i+1}$ is injective for $i=0,1$, and when $i=2$ the
rank is at least $5$.

\begin{proof}
First, it is clear that the only question arises when $i=2$, since for $i=0,1$, the ring $A$ coincides with the polynomial ring.  Let $A = S/I$ and let $J = \frac{I+(L)}{(L)}$, which we will view by abuse of notation as an ideal in $k[x,y]$. We have the exact sequence
\[
[A]_2 \stackrel{\times L}{\longrightarrow} [A]_3 \rightarrow [k[x,y]/J]_3 \rightarrow 0.
\]
Since $6 = \binom{4}{3} + \binom{2}{2} + \binom{1}{1}$, it follows from a theorem of Green (cf. \cite[Theorem 1]{green}) that $\dim [k[x,y]/J]_3 \leq 1$, i.e. that $\times L$ has rank $\geq 5$ as claimed.
\end{proof}

\end{lemma}

We now consider a $4$-dimensional vector subspace, $W = \langle f_1,f_2,f_3, f_4\rangle$, of cubics in $S$.  Assume that $f_1,f_2,f_3,f_4$ have no common zeros in $\mathbb P^2$, so they define an artinian ideal.  Associated to these polynomials is a morphism $\Phi : \mathbb P^2 \rightarrow \mathbb P^3$.   $\Phi$ is either generically one-to-one or generically three-to-one, and correspondingly $\Phi(\mathbb P^2)$ either has degree 9 or 3. We will be interested in pencils defined by $2$-dimensional subspaces of $W$.  Notice that it is equivalent to choose a general $2$-dimensional subspace and consider its base locus, $Y$, which will be a complete intersection, or to choose a general line, $\lambda$, in $\mathbb P^3$, intersect it with $\Phi(\mathbb P^2)$, and look at the pre-image of the intersection points.  We will use both points of view.

\begin{remark} \label{exist?}
One of the consequences of our work will be that if $R/( f_1, f_2, f_3,f_4)$ fails the WLP then $\Phi$ cannot be generically one-to-one, except possibly in characteristic 3 (see Proposition \ref{answer ref 2}).  However, in order to finally reach this conclusion, we have to begin by seeing some intermediate consequences of such a property.  

On the other hand, the well-known example $W=\langle x^3, y^3, z^3, xyz \rangle$ shows that $\Phi$ {\em can} be generically three-to-one.  
\end{remark}

\begin{proposition} \label{hesse}
Let $f_1, f_2, f_3, f_4$ be independent forms of degree $3$ defining
an artinian ideal, $J$.  Assume that $S/J$ fails the WLP from degree
$2$ to degree $3$. Let $\mathcal L$ be the pencil of cubics in
$\mathbb P^2$ defined by a general $2$-dimensional subspace of
$W=\langle f_1,f_2,f_3,f_4\rangle$ and let $Y$ be the scheme-theoretic
base locus of $\mathcal L$.

\begin{itemize}

\item[(a)] If $\Phi$ is generically one-to-one, then $Y$ is a reduced Hesse configuration in $\mathbb P^2$. 

\item[(b)] If $\Phi$ is generically three-to-one, then $Y$ decomposes
as  $\Sigma_1 \cup \Sigma_2 \cup \Sigma_3$, where
$Q_1=\Phi(\Sigma_1)$,
$Q_2=\Phi(\Sigma_2)$ and
$Q_3=\Phi(\Sigma_3)$ are distinct points. Furthermore, any line in
$\mathbb P^2$ joining points from two of these parts of $Y$ will also
contain a point from the third part. 
For each $i$, $\Sigma_i$ is a scheme of degree $3$ that imposes only one condition on the elements of $W$.

\end{itemize}
\end{proposition}

\begin{proof}

Note that the condition that $\mathcal L$ is {\em general} implies
that its base locus, $Y$, is zero-dimensional, and that $Y$ being zero
dimensional  is equivalent to the condition that $Y$ is a complete
intersection (scheme-theoretically).  

Since $\mathcal L$ corresponds
to a general line in $\mathbb P^3$, the image of $Y$ under $\Phi$
consists of $\deg(\Phi(\mathbb P^2))$ distinct points. 
  In characteristic $0$ this is a consequence of Bertini's theorem,
  but in characteristic $p$ we need a separate argument. Let $h$ be
  the defining polynomial of the image $\Phi(\mathbb P^2)$ in ${\mathbb P}^3$. Since $\Phi(\mathbb P^2)$ is the image of an irreducible variety, we have that $h$ is an irreducible polynomial. 

  The tangent space to $\Phi(\mathbb P^2)$ at a point is $2$-dimensional, unless $\partial   h/\partial x_0 = \partial h/\partial x_1= \partial h/\partial x_2   = \partial h/\partial x_3 = 0$. The singular locus is given by the   points where all partial derivatives vanish. This is a subvariety   of codimension at least $1$ unless all partials vanish completely   on $\Phi(\mathbb P^2)$, but that means that they are identically   zero, since $f$ does not have any nontrivial factors. Thus we get   that all variables occur only as $p$th powers, i.e., $h$ is a   polynomial in $x_0^p,x_1^p,x_2^p$. Since the characteristic of   $k$ is $p$, we hence get that $h$ is a $p$th power of another form, contradicting the irreducibility of $h$. 

  The tangent lines are either lines through a singular point of $\Phi(\mathbb P^2)$, which gives us a variety of dimension at most $1+2=3$ in the $4$-dimensional grassmannian of lines in ${\mathbb P}^3$, or lines in the $2$-dimensional tangent spaces at the nonsingular points, which gives a family of dimension $2+1=3$. Thus, the general line is not tangent to $\Phi(\mathbb P^2)$, and hence it meets $\Phi(\mathbb P^2)$ in $\deg(\Phi(\mathbb P^2))$ distinct points by Bezout's theorem.

Let $\mathcal L, W$ and $Y$ be as in the statement of (a).  
Since $\mathcal L$ corresponds to the intersection of $\Phi(\mathbb P^2)$
with a general line, $\lambda$, we can assume that $\Phi(Y)$ consists
of nine distinct points and hence $Y$ is reduced.
 
Let $Z = \{P_1,P_2\}$ be a set of two points of $Y$, so the $h$-vector of $S/I_Z$ is $(1,1)$.  There are two possibilities, a priori: either (i) $Z$ imposes one condition on $W$ or (ii) $Z$ imposes two conditions. But in fact since one can find a hyperplane in $\mathbb P^3$ containing $\Phi(P_1)$ but not $\Phi(P_2)$, $Z$ imposes two conditions.  Thus $Z$ uniquely determines $\mathcal L$.  

Let $\ell$ be the linear form in the ideal of $Z$.  Consider the ideal
$(\ell, I_Y)$ which needs not be saturated.  This defines a zero-dimensional scheme, which we want to show has length exactly 3.  First note that the length must be at most 3, since it is a zero-dimensional scheme cut out by cubics. Suppose that $(\ell, I_Y)$ defines a scheme of length 2, namely $Z$.  Then by liaison, $I_Y : (\ell, I_Y) = I_Y : (\ell) := I_{Z'}$ defines a scheme, $Z'$, of length $7$ not lying on a conic, with $h$-vector $(1,2,3,1)$.  On the other hand, since $S/J$ fails the WLP, there is a form $q$ of degree $2$ such that $\ell \cdot q \in (f_1,f_2,f_3,f_4)$.  Since $Z$ determines $\mathcal L$ and $\ell \cdot q \in I_Z$, we have $\ell \cdot q \in I_Y$ and  $q \in I_{Z'}$.  Contradiction.  Thus $(\ell, I_Y)$ defines a scheme of length $3$ and $Z'$ lies on a conic, since its $h$-vector is $(1,2, 2,1)$ (again by liaison).  Since $Z$ was an arbitrary choice of two points of $Y$, we are done (a).

Now consider case (b).  As before, for a general line $\lambda$ we can
assume that $\lambda$ meets $\Phi(\mathbb P^2)$ in three distinct
points, $Q_1,Q_2,Q_3$. Hence the pre-image of these three points is a
disjoint union $Y = \Sigma_1 \cup \Sigma_2 \cup \Sigma_3$, where
$\Sigma_i = \Phi^{-1} (Q_i)$ is a scheme of length $3$ and $Y$ is still
a complete intersection.  Let $P_i$ and $P_j$ be reduced points in
$\Sigma_i$ and $\Sigma_j$ respectively ($i \neq j$).  The same
argument as above shows that the line joining $P_i$ and $P_j$ meets
$Y$ in a third point. We have to argue that it cannot be a point of
$\Sigma_i$ or $\Sigma_j$. Suppose it was. Then the plane cubic curve,
which is the image of the line joining $P_i$ and $P_j$,
would meet the line $\lambda$ at only two distinct points. This
means that $\lambda$ is tangent to the curve, but hence also tangent
to $\Phi(\mathbb P^2)$. However, as we saw above, a general line will
not be tangent to $\Phi(\mathbb P^2)$. Thus we conclude that the line
contains points from each of the three parts.   

The  pre-image of $Q_i$ can be found by choosing three general hyperplanes through $Q_i$ and considering the intersection in $\mathbb P^2$ of the three elements of $W$ corresponding to these hyperplanes.  Let $P \in \Sigma_i$ be a reduced point. The fact that $\Sigma_i$ is the pre-image of $Q_i$ means that any element of $W$ vanishing at $P$ (which is a codimension $1$ condition) in fact vanishes on all of $\Sigma_i$.  Thus $\Sigma_i$ imposes only one condition.
\end{proof}

In the proof of the main theorem, we will use the fact that complete
intersections of type $(2,3,3)$ have the WLP. This is known in
characteristic zero, but here we provide a complete characterization
in any characteristic. 

\begin{lemma}\label{233}
If $k$ is an infinite field, any complete intersection with Hilbert function $(1,3,5,5,3,1)$ has the WLP, except  when the characteristic of the field $k$ is $3$ and the ideal is $I=(x^2,y^3,z^3)$, after a change of variables. 
\end{lemma}

\begin{proof}
  Let $I=(q,f,g)\subseteq S=k[x,y,z]$ be a complete intersection in
  $S$, where $\deg q=2$ and $\deg f=\deg g=3$. Now $g$ is a non-zerodivisor on $X$ defined by $(q,f)$. If the multiplication by a linear
  form $\ell$ is not injective on $A=S/I$ from degree $2$ to degree
  $3$ we have that $\ell q_1 = \ell_1 q + \alpha f + \beta g$, for
  some quadratic form $q_1$. This means that $\beta=0$ if $\ell$ is a zerodivisor on $X$. Thus, in order to show that we have the WLP, it suffices to find a linear form $\ell$ which is a zerodivisor on $X$ but still have injective multiplication from degree $2$ to degree $3$ on the coordinate ring $B=S/(q,f)$ of $X$. 

We will first show that for a general form $h=\alpha f + \beta g$, the subscheme defined by $(q,h)$ has at least one simple point, unless $q$ is a square or the characteristic of $k$ is $3$. First assume that $q=\ell_1\ell_2$ and we get $(q,h) = (\ell_1,h)\cap(\ell_2,h)$. On each of the lines, we get a simple point, unless $h$ is a cube. However, in the polynomial ring $k[x,y]$, a general linear combination of two cubes $\alpha x^3+\beta y^3$ is a cube if and only if the characteristic of $k$ is $3$. If we specialize from a nonsingular conic to a union of two lines, we get at least the same multiplicities in the limit as in the general elements. Thus, unless we are in characteristic $3$, we have a simple point of the general complete intersection $(q,h)$. 

Let $\ell$ be a linear form which meets $X=V(q,f)$ in a simple point only. Then the residual 
subscheme $Y$ given by $(q,f)\colon \ell$ has $h$-vector $(1,2,2)$ by linkage. Thus $q$ is the only quadric vanishing on $Y$, which shows that multiplication by $\ell$ is injective from degree $2$ to degree $3$ on the coordinate ring of $X$ and since it is a zerodivisor on $X$ it is also injective on $A=S/(f,g,h)$. 

Now, suppose that $q$ is a square. We can assume that $q = x^2$ after a change of variables and that $(x,y)$ defines a point where $(x,f)$ is not a double point. We can write $f=a_0xy^2+a_1xyz+a_2xz^2+a_3y^3+a_4y^2z+yz^2$, since we know that the coefficient of $yz^2$ is non-zero as $(x,y)$ is not a double point on $(x,f)$. We look at the multiplication by the linear form $\ell = \alpha x+ \beta y$, for $\alpha,\beta\in k$. We can choose bases $\{xy,xz,y^2,yz,z^2\}$ and $\{xy^2,xyz,xz^2,y^3,y^2z,z^3\}$ for $B_2$ and $B_3$ since $f$ allows us to solve for $yz^2$. Then the matrix corresponding to multiplication by $\ell$ will be
$$
\begin{bmatrix}
  \beta & 0 & 0 & 0 & 0 &0 \\
  0 & \beta & 0 & 0 & 0 & 0\\
  \alpha  & 0 & 0 & \beta & 0 &0\\
  0 & \alpha & 0 & 0 & \beta &0\\
  -a_0\beta & -a_1\beta & \alpha-a_2\beta & -a_3\beta & -a_4\beta &0\\
\end{bmatrix}
$$
with the first maximal minor equal to $\beta^4(\alpha-a_2\beta)$. Thus
for $\beta\ne0$ and $\alpha\ne a_2\beta$, the multiplication by $\ell$
is injective on $S/(q,f)$ from degree $2$ to degree $3$. Since $\ell$
is a zerodivisor on $X$, we conclude that $\ell$ is also injective on $S/(q,f,g)$.

We now turn to the remaining case, when the characteristic of $k$ is $3$ and $f,g$ are cubes of linear forms, and suppose that $q$ is not a power of a linear form. We can assume that $f=y^3$ and $g=z^3$. Furthermore, we may assume that $(x,y)$ defines a point on $q=0$ other than the double point of $q$. Since $(q,f,g)$ is a complete intersection, the coefficient of $x^2$ in $q$ has to be non-zero, so we may write $q=x^2+a_1xy+a_2xz+a_3y^2+a_4yz$. Now we choose bases $\{xy,xz,y^2,yz,z^2\}$ and $\{xy^2,xyz,xz^2,yz^3,y^2z,z^3\}$ for $A_2$ and $A_3$ and the multiplication homomorphism for the linear form $\ell = \alpha x+ \beta y$ is given by 
$$
\begin{bmatrix}
  \beta -\alpha a_1 & -\alpha a_2 & 0 & -\alpha a_4 & 0 & 0 \\
  0 & \beta - \alpha a_1 & -\alpha a_2 & -\alpha a_3 & -\alpha a_4 & 0 \\
\alpha & 0 & 0 & 0 & 0 &0 \\
0&\alpha & 0 & \beta & 0 &0 \\
0&0&\alpha & 0 & \beta &0 \\
\end{bmatrix}.
$$
The first maximal minor of this matrix equals $\alpha^3(a_4\alpha -a_2\beta)^2$ which is non-zero for general $\alpha,\beta$ unless $a_2=a_4=0$. If $a_2=a_4=0$, we have that $q$ is a product of two linear forms vanishing at $[0:0:1]$, contradicting the assumption that $(x,y)$ is a smooth point of $q$. Hence we conclude that the only possibility for the WLP to fail is when the characteristic of $k$ is $3$ and  $I=(x^2,y^3,z^3)$ up to a change of variables, which indeed fails the WLP in characteristic 3.
\end{proof}

\begin{theorem}\label{MainThm}
Any artinian Gorenstein
  algebra $S/I$ with Hilbert function $(1,3,6,6,3,1)$ has the WLP unless the characteristic of $k$ is $3$ and the ideal $I = (x^2y, x^2z, y^3, z^3, x^4+y^2z^2)$,  after a change of variables. 
\end{theorem}

\begin{proof}
  Let $A=S/I$ be an artinian Gorenstein algebra with Hilbert function
  $(1,3,6,6,3,1)$.     By \cite[Theorem 2.1]{BE}, the ideal of any
  Gorenstein algebra of codimension $3$ is given by the pfaffians of
  a skew-symmetric matrix. In this case, we either have a $5\times
  5$-matrix, or a $7\times 7$-matrix.  In the latter case, the matrix
  can be written as 
  $$
  \begin{bmatrix}
    P&N\\-N^t&0
  \end{bmatrix}
  $$
  where $P$ is a $4\times 4$-matrix of quadrics and $N$ is a $4\times
  3$-matrix of linear forms. By \cite[Lemma 2.12]{MZ}, the four cubics
  in the ideal $I$  have no GCD of degree $>0$. These four cubics in
  $I$ are given by the maximal minors of $N$ and hence define a
  codimension $2$ Cohen-Macaulay ring. Thus it is clear that the
  multiplication by a non-zerodivisor gives injective multiplication
  also on $A$ up to degree $2$. 

  Now look at the case of a $5\times 5$-matrix. Denote the five
  pfaffians by $f_1,f_2,f_3,f_4,f_5$ in the natural order.   Let
  \begin{equation}\label{BE}
  Q := 
  \begin{bmatrix}
    0 & q_1 & q_2 & q_4 & \ell_1 \\
    -q_1 & 0 & q_3 & q_5 & \ell_2 \\
    -q_2 & -q_3 & 0& q_6 & \ell_3 \\
    -q_4 & -q_5 & -q_6 & 0 & \ell_4 \\
    -\ell_1 & -\ell_2 &-\ell_3 & -\ell_4 & 0 \\
  \end{bmatrix}.
  \end{equation}
  Since we are in a polynomial ring $S$ with only three variables, we
  can eliminate $\ell_4$ by means of row and column operations, keeping
  the matrix skew-symmetric. 
  
  Let $J$ be the ideal defined by the first three cubics. These are the
  maximal minors of the upper right $3\times 2$-block. Hence they define
  a codimension $2$ Cohen-Macaulay algebra of multiplicity $7$ and
  Hilbert function $(1,3,6,7,7,7,\dots)$ unless there is a common factor
  among them. If there were a common factor, $f$, we see from the colon
  ideal $(I\colon f)$ that we get a Gorenstein quotient with three
  forms, $f_1/f,f_2/f,f_3/f$, of degree $1$ or $2$ in the ideal
  $(I\colon f)$. In degree $1$ the  three forms would generate
  everything forcing $(I\colon f)=\mathfrak m$. If we had three
  quadrics, the Hilbert function of $S/(I\colon f)$ would be
  $(1,3,3,3,1)$, but then there would be another linear syzygy among
  the three forms $f_1,f_2,f_3$, contradicting that the matrix $Q$
  gave all the syzygies of the forms $f_1,f_2,f_3,f_4,f_5$. Thus we 
  conclude that we get a zero-dimensional subscheme, $X$, of length $7$
  defined by the ideal $J$.

  If the linear forms $\ell_1,\ell_2,\ell_3$ are linearly
  independent, we have that the quartic generator, $f_5$, satisfies
  $f_5{\mathfrak m}\subseteq (f_1,f_2,f_3,f_4)$. Thus $f_5$ is a socle
  element in $B=S/(f_1,f_2,f_3,f_4)$, so since $A$ is artinian, also $B$
  must be artinian. Suppose that WLP does not hold.  By Proposition~\ref{tfae}, 
 since the linear forms are independent, the length $7$ scheme $X$ contains no subscheme of length $6$ lying on a
  conic.   
  
  On the other hand, let $P_1$ and $P_2$ be general (in particular distinct) points in $\mathbb P^2$.  These points impose independent conditions on the vector space spanned by $f_1,f_2,f_3,f_4$, and so define a pencil, hence a complete intersection, $Y$.  Let 
\[
\Phi = [f_1,f_2,f_3,f_4] : \mathbb P^2 \longrightarrow \mathbb P^3
\]
be defined as in Proposition~\ref{hesse}.  We  assume from now on that the ideal generated by $f_1,f_2,f_3,f_4$ fails the WLP from degree $2$ to degree $3$.
  
Assume first that $\Phi$ is generically one-to-one.  By Proposition~\ref{hesse}, this complete intersection is a reduced Hesse configuration, $\mathcal H$. Thus the line spanned by $P_1$ and $P_2$ contains a third point of $\mathcal H$.  It follows that the residual to $P_1 \cup P_2$ is a length $7$ subscheme with the same Betti numbers as $X$, containing a length $6$ subscheme lying on a conic.  Indeed, the residual to {\em any} length $2$ subscheme of $\mathcal H$ is a length $7$ subscheme containing a subscheme of length $6$ lying on a conic.  
Since the ideal of $\mathcal H$ can be viewed as a point of the Grassmannian $G(2,4)$, we can specialize $\mathcal H$ to a complete intersection containing $X$.  By semicontinuity, $X$ also 
has the property that it contains a length $6$ subscheme lying on a conic.  Contradiction. 

Now assume that $\Phi$ is generically three-to-one.  It follows that $\deg \Phi(\mathbb P^2) = 3$.  If $P$ is a general point of $\mathbb P^2$, the $3$-dimensional subspace of $W$ determined by the vanishing at $P$ in fact also vanishes on a length $3$ subscheme $\Sigma_P$, as described in Proposition~\ref{hesse}.  Now, the ideal of $X$ is generated by three cubics, and we can assume that they are $f_1,f_2,f_3$.  Thus $X$ is the pre-image of the point $[0:0:0:1]$ under $\Phi$.  

We will now prove that $\Phi(\mathbb P^2)$ is not a cone with vertex at the
point $ [0:0:0:1]$.   Suppose that $\Phi(\mathbb P^2)$
were a cone with vertex at $P=[0:0:0:1]$. Then the general
hyperplane through $[0:0:0:1]$ cuts $\Phi(\mathbb P^2)$
in a union of three distinct lines according to Proposition~\ref{hesse}. Since the inverse image $W$ of this hyperplane section is given by a general form $\alpha f_1+ \beta f_2 + \gamma f_3$ of the linear system $\langle f_1,f_2,f_3\rangle$, it has to be a union of three distinct lines in ${\mathbb P}^2$. The intersection between any two of these lines has to map to $[0:0:0:1]$, so it is a point in the support of $X$.  Since there are only finitely many such line configurations meeting at three distinct points of $X$, we would get that the general point of $\Phi(\mathbb P^2)$ had an inverse image lying on the union of a finite number of lines, unless the general hyperplane section gives three lines in ${\mathbb P}^2$ through a single point. However, this means that the general form $\alpha f_1+ \beta f_2 + \gamma f_3$ is a cubic in two linear forms, which contradicts the assumption that there is a linear syzygy among $f_1,f_2,f_3$ involving all three variables $x,y,z$. Thus $\Phi(\mathbb P^2)$ cannot be a cone with vertex at $[0:0:0:1]$. 

Choosing a general point $P \in \Phi(\mathbb P^2)$, the  line joining $P$ and $[0:0:0:1]$ meets $\Phi (\mathbb P^2)$ in zero-dimensional scheme, so the pre-image of this scheme is a complete intersection of type $(3,3)$ containing $X$.  However, the pre-image of $P$ is a length $3$ scheme disjoint from $X$.  Since $X$ has length $7$, this is impossible.  

We conclude that if the linear forms $\ell_1,\ell_2, \ell_3$ in the
Buchsbaum-Eisenbud matrix are linearly independent then $S/I$ must
have the WLP, regardless of the characteristic of the field $k$. 

 Now assume that $\ell_1,\ell_2,\ell_3$ are linearly
  dependent. Then they must span a $2$-dimensional vector space, since a
  syzygy must involve at least two generators. Hence we can assume that
  also $\ell_3=0$. From the skew-symmetric matrix, we can see that
  $\ell_4=\ell _3=0$ implies that $(f_1,f_2,f_3,f_4)\subseteq
  (\ell_1,\ell_2)$. Hence we must have that the saturation of
  $(f_1,f_2,f_3,f_4)$ is equal to $(\ell_1,\ell_2)$, so in particular
  $B=S/(f_1,f_2,f_3,f_4)$ is not artinian. Now $f_5$ is not a socle
  element, but since $(\ell_1,\ell_2)f_5\subseteq (f_1,f_2,f_3,f_4)$, we
  have that the Hilbert function of the ideal $(f_5)$ in $B$ must be
  $(0,0,0,0,1,1,1,\dots)$. Thus the Hilbert function of $B$ must be
  $(1,3,6,6,4,2,1,1,1,\dots)$. Let $Y$ be the subscheme of ${\mathbb
    P}^2$ defined by $(\ell_1,\ell_2)$ and let $X$ be the subscheme of
  ${\mathbb P}^2$ defined by $J=(f_1,f_2,f_3)$. The ideal generated by
  $f_4$ in $R_X=S/J$ has Hilbert function $(0,0,0,1,3,5,6,6,6,\dots)$
  and the cyclic module corresponds to the coordinate ring of a complete
  intersection defined by $(q_6,f_3)$. Denote this subscheme by $Z$. We
  now have that $X=Z\cup Y$ as subschemes and we have an inclusion 
  $$
  R_X \longrightarrow R_Y \oplus R_Z
  $$
  which is an isomorphism from degree $2$ and higher because the
  Hilbert functions agree in these degrees. The form $f_4$ is zero on
  $R_Y$ and hence a non-zerodivisor on $R_Z$. Thus we get an
  injection 
  $$
  B \longrightarrow R_Y \oplus R_Z/(f_4).
  $$
  By Lemma~\ref{233}, $R_Z/(f_4)$ has the WLP except when the
  characteristic of $k$ is $3$ and after change of variables,
  $(q_6,f_3,f_4) = (x^2,y^3,z^3)$.  Clearly $R_Y$ always has the WLP.  Thus $B$, and hence $A$, has the WLP except when  the characteristic of $k$ is $3$ and after change of variables, $(q_6,f_3,f_4) = (x^2,y^3,z^3)$.
 
Since the cubics in the Gorenstein ideal are given by $\ell_1 q_6,\ell_2 q_6,\ell_1q_3-\ell_2q_2,\ell_1q_5-\ell_2q_4$, we have to have $q_6=x^2$ and $(\ell_1,\ell_2)=(y,z)$. This means that the cubics in the ideal are $x^2y,x^2z,y^3,z^3$.  We can find the quartic generator by computing the appropriate Pfaffian from (\ref{BE}), using $q_3=y^2, q_4=z^2, q_2=0,q_5=0,q_1=x^2$, and we obtain $f_5 = x^4+y^2z^2$ as desired.  Indeed, a direct calculation shows that this fails the WLP in characteristic $3$.
\end{proof}

The proof of the above theorem shows that in the general case (when the skew-symmetric matrix has size $5 \times 5$), three of the four minimal generators of degree $3$ can be chosen as the saturated ideal of a zero-dimensional scheme of length $7$.  The question of whether $\ell_1, \ell_2, \ell_3$ are independent is not related to whether or not a suitable Gorenstein algebra exists (since we assume that we begin with a Gorenstein algebra), but rather whether the degree $3$ component is artinian (see also Proposition~\ref{tfae}). That is, in this situation the general behavior is for the fourth cubic generator to be a non-zerodivisor on the coordinate ring of this length $7$ scheme. In the next example we show that in either case, this length $7$ scheme can even be curvilinear (i.e. lying on a smooth curve) and supported at a point.

\begin{example} \label{non red}
First consider the codimension $3$ Gorenstein ideal generated by the Pfaffians
of the following matrix: 
\[
\left [
  \begin{array}{cccccc}
    0  &        0 &      0  &    x^2 &  y \\
    0  &        0  &   x^2 &   y^2 &  z \\
    0   &    -x^2 &   0  &    z^2 &  0 \\
    -x^2 & -y^2 &  -z^2  &   0 &   0 \\
    -y &    -z &    0  &      0 &   0 
  \end{array}
\right ]
\]
The length $7$ scheme is generated by the maximal minors of  the matrix
\[
\left [
\begin{array}{cccc}
x^2 &  y \\
y^2  & z \\
z^2 &  0
\end{array}
\right ].
\]
The ideal is $( x^2z - y^3, yz^2, z^3 )$.  It is clearly supported at
the point $[1:0:0]$.  However, notice that the first generator is
smooth at this point, hence this is a curvilinear scheme. 

On the other hand, using a computer algebra program, one can verify that if $P$ is a general point on a general cubic curve in $\mathbb P^2$, then the divisor $7P$ defines a curvilinear zero-dimensional scheme supported at $P$ with the desired Betti numbers, having three independent linear forms. 
\end{example}

\smallskip

We now return to the question of whether the situation in part (a) of Proposition \ref{hesse} can occur (see Remark \ref{exist?}).  We summarize our answer:

\begin{itemize}
\item[(a)] If $\hbox{char}(k) \neq 3$ then  under the assumption that WLP fails, $\Phi$ is never generically one-to-one.

\item[(b)] If $\hbox{char}(k) = 3$ {\em and} $I$ is the degree 3 part of a Gorenstein ideal with Hilbert function $(1,3,6,6,3,1)$ then again  under the assumption that WLP fails, $\Phi$ is never generically one-to-one.

\item[(c)] If $\hbox{char} (k) = 3$ and $I$ is {\em not} the degree 3 part of a Gorenstein ideal with Hilbert function $(1,3,6,6,3,1)$ then under the assumption that WLP fails, we do not know if $\Phi$ can be generically one-to-one or not, but we believe not.
\end{itemize}

Assertion (b) is just the conclusion of Theorem \ref{MainThm},
together with the observation that the cubics in the ideal given in
the statement of the theorem do not generate an artinian ideal nor
do they define a generically one-to-one morphism.  Assertion (c) needs no comment.

Thus we have only to show assertion (a).  We will use the term
\emph{Hesse pencil} for a pencil of cubics whose base locus is a Hesse
configuration, which is slightly different than the use in \cite{AD}. 

\begin{proposition} \label{AD prop}
Fix the variables $x,y,z$.  Assume $\hbox{char}(k) \neq 3$.  Let $H$
be a reduced Hesse pencil containing the polynomial $xyz$.  Then
$H=\langle ax^3 + by^3 + cz^3, xyz\rangle$ for some non-zero constants $a,b,c$. 
\end{proposition}

\begin{proof}
This follows from the discussion \cite[Section 2]{AD}.
\end{proof}

\begin{proposition} \label{answer ref 2}
Let $f_1, f_2, f_3, f_4$ be independent forms of degree $3$ defining an
artinian ideal $J$ and assume $S/J$ fails WLP from degree $2$ to
degree $3$. Assume $\hbox{char}(k) \neq 3$.  Then $\Phi$ is not
generically one-to-one.  
\end{proposition}

\begin{proof}
Suppose it were.  Then  there is at most a curve, $C$ in $\mathbb P^2$
where $\Phi$ fails to be one-to-one.  By Proposition \ref{hesse} (a),
a general pencil $\langle g_1,g_2\rangle$ in $\langle f_1,f_2,f_3,f_4
\rangle$ is a reduced Hesse pencil $H_1$.  Since the pencil is
general, we can assume that none of the base points of $H_1$ lies on
$C$.  (Choose $g_1$ and then choose $g_2$ avoiding the points of
$V(g_1) \cap C$.)  So the nine base points of $H_1$ go to nine
different points under $\Phi$. 

By Proposition \ref{AD prop} we can write $g_1 = x^3 + y^3 + z^3$ and
$g_2 = xyz$ by changing variables.  Since none of the base points of $H_1$
lie on $C$, neither $x, y$ nor $z$ defines a component of $C$.  So for a
general choice of $g_3$ in the vector space $\langle
f_1,f_2,f_3,f_4\rangle$ we have that $(g_3,xyz)$ gives a reduced 
complete intersection all of whose points avoid $C$.  Thus this
complete intersection is also a reduced Hesse pencil, $H_2$, and by
Proposition \ref{AD prop} we have $H_2 = \langle xyz, a_2 x^3+b_2
y^3+c_2 z^3 \rangle$.  Doing this once more we produce one more
reduced Hesse pencil $H_3=\langle xyz, a_3
x^3+b_3 y^3+c_3 z^3\rangle$.   By the general choice of the $g_i$, the three
cubics $g_1, g_2, g_3$ are linearly independent.  Thus taking linear
combinations, we get that  $\langle f_1,f_2,f_3,f_4 \rangle = \langle
x^3,y^3,z^3,xyz \rangle$.  Thus $\Phi$ is not generically one-to-one.
Contradiction. 
\end{proof}

Putting together some of the previous results, we have at once:

\begin{corollary}\label{6} In characteristic zero, all codimension $3$ artinian Gorenstein algebras of socle degree at most $6$ enjoy the WLP.
\end{corollary}

\begin{proof}
From Theorem~\ref{MainThm}, Lemma~\ref{lemma:odd}, and Corollary~\ref{reductiontocompressedodd}, we  immediately obtain the  result for socle degrees $5$ and $6$. Since the corollary is known to hold true in socle degree $3$ (see \cite[Corollary 3.2]{MZ} or \cite[Theorem 3.3]{BZ}), by Corollary~\ref{reductiontocompressedodd} it also   follows for socle degree $4$, and the proof is complete.
\end{proof}
 
\begin{corollary}\label{cor:SLP} In characteristic zero, all codimension
  $3$ artinian Gorenstein algebras of socle degree at most $5$ have the
  SLP. 
\end{corollary}

\begin{proof}
  For socle degree $5$,  we only have to verify the maximal rank in the remaining two symmetric cases $\times
 L^3\colon [A]_1\longrightarrow [A]_4$ and $\times
 L^5\colon [A]_0\longrightarrow [A]_5$. The latter is bijective since not all fifth powers of linear forms can be in the
 ideal. 

 The first is bijective due to a result by Gordan and Noether
 stating that a ternary form with vanishing Hessian has to reduce to a
 binary form after a linear change of coordinates (see
 \cite{gordan-noether} and \cite[Theorem~30]{maeno-watanabe}). For
 lower socle degrees, the argument is exactly the same with a shift of
 degrees. 
\end{proof}

\begin{remark}\label{Partition}
  An even finer study of the multiplications by powers of linear forms
  include the various Jordan partitions that we may get for
  multiplication by linear forms on the artinian algebra considered as
  a $k$-vector space. In particular, we look at the case where the WLP
  fails for artinian Gorenstein algebras with $h$-vector
  $(1,3,6,6,3,1)$. In this case, the ideal after a change of variables
  is given by $I=(x^2y,x^2z,y^3,z^3,x^4+y^2z^2)$. We get partition $(6,3,3,3,3,1,1)$ for a general
  linear form. There are only two other possibilities, one for
  multiplication by $x$ where we get $(6,2,2,2,2,2,2,2)$ and one for
  multiplication by $ay+bz$ where we get $(3,3,3,3,3,3,1,1)$. 

  In characteristic zero, a general linear form will have the
  partition $(6,4,4,2,2,2)$, because of the SLP.
\end{remark}

\begin{remark}\label{linkage}
  In our proof of the main theorem, the length $7$ subscheme
  obtained from a submatrix of the Buchsbaum-Eisenbud matrix played a
  central role. In fact, we can in general eliminate even further and
  get the matrix to be of the form 
$$
 \begin{bmatrix}
    0 & 0 & q_2 & q_4 & \ell_1 \\
    0 & 0 & q_3 & q_5 & \ell_2 \\
    -q_2 & -q_3 & 0& q_6 & \ell_3 \\
    -q_4 & -q_5 & -q_6 & 0 & 0 \\
    -\ell_1 & -\ell_2 &-\ell_3 & 0 & 0 \\
  \end{bmatrix}.$$
  In this way, we see that the length $7$ subscheme $X$ is linked to
  a length $8$ subscheme $Y$ defined by the maximal minors of the
  matrix
  $$
  \begin{bmatrix}
  q_2 & q_4 & \ell_1 \\
  q_3 & q_5 & \ell_2 \\
  \end{bmatrix}
  $$
 via the complete intersection of type $(3,5)$ given by the common
 cubic and the determinant of the submatrix 
$$
 \begin{bmatrix}
  q_2 & q_4 & \ell_1 \\
  q_3 & q_5 & \ell_2 \\
  0& q_6 & \ell_3 \\
\end{bmatrix}.$$
This shows that all artinian Gorenstein algebras with Hilbert function
$(1,3,6,6,3,1)$ and five minimal generators can be obtained as
$S/(I_X+I_Y)$ where $X$ and $Y$ are as described above. 
\end{remark}


 \section*{Acknowledgements}    The authors would like to thank the
 University of Notre Dame, CIRM (Trento) and CRM (Barcelona) for kind
 hospitality and generous support. In the course of writing this
 paper, the authors made use of the computer programs Macaulay2
 \cite{macaulay2} and \cocoa \cite{cocoa}. The picture in
 Remark~\ref{hesse picture} was obtained from {\tt
   wikipedia.org}. The authors thank David Cook II, Brian Harbourne, Anthony
 Iarrobino and Junzo Watanabe for useful discussions on the topic of
 this paper.

%
%

\bibliographystyle{abbrv}

\end{document}